\documentclass[pdflatex,sn-mathphys-num]{sn-jnl}% Math and Physical Sciences Numbered Reference Style
\usepackage{graphicx}%
\usepackage{multirow}%
\usepackage{amsmath,amssymb,amsfonts}%
\usepackage{amsthm}%
\usepackage{mathrsfs}%
\usepackage[title]{appendix}%
\usepackage{xcolor}%
\usepackage{textcomp}%
\usepackage{manyfoot}%
\usepackage{booktabs}%
\usepackage{algorithm}%
\usepackage{algorithmicx}%
\usepackage{algpseudocode}%
\usepackage{listings}%
\usepackage{subfigure}
\usepackage{diagbox}
\usepackage[final]{changes}

\newtheorem{theorem}{Theorem}[section]
\newtheorem{definition}{Definition}[section]
\newtheorem{remark}{Remark}[section]

\newtheorem{proposition}{Proposition}[section]
\newtheorem{example}{Example}[section]

\usepackage{booktabs}
\usepackage{siunitx}

\newcommand{\p}{\partial}

\newcommand{\pd}[2]{\frac{\partial #1}{\partial #2}} 
\newcommand{\pdd}[2]{\frac{\partial^2 #1}{\partial #2^2}} 
\newcommand{\R}{\mathbb{R}}
\newcommand{\T}[1]{\text{#1}}
\newcommand{\vl}{\Big|}
\begin{document}

\title[Article Title]{Leveraging Lie Group Symmetries to Enhance Physics-Informed Neural Networks for the Fundamental Solution of Linear PDEs}


\author[1, 2]{\fnm{Xiaopei} \sur{Jiao}}\email{jiaoxiaopei@bimsa.cn}

\author[2]{\fnm{Fansheng} \sur{Xiong}}\email{fansheng\_xiong@bimsa.cn}

\affil[1]{\orgdiv{Department of Mathematical Science}, \orgname{Tsinghua University}, \orgaddress{\city{Beijing}, \postcode{100084}, \country{China}}}

\affil[2]{\orgname{Beijing Institute of Mathematical Sciences and Applications (BIMSA)}, \orgaddress{\city{Beijing}, \postcode{101408 }, \country{China}}}

%\affil[3]{\orgdiv{Department}, \orgname{Organization}, \orgaddress{\street{Street}, \city{City}, \postcode{610101}, \state{State}, \country{Country}}}


\abstract{Since the introduction of deep learning for solving partial differential equations (PDEs), there has been growing interest in real-time system responses, where the kernel function plays a key role. Physics-informed neural networks (PINNs), a popular mesh-free, semi-supervised learning tool, offer high flexibility. This paper explores the integration of Lie symmetry groups with deep learning techniques to enhance the numerical solutions of fundamental PDEs. We propose a novel approach that combines PINNs and Lie group theory to address computational inefficiencies in traditional methods. By incorporating the linearized symmetric condition (LSC) derived from Lie symmetries into PINNs, we introduce a new residual loss function that requires fewer derivatives for calculation. This integration reduces computational costs and improves solution accuracy. {Numerical simulations demonstrate a significant reduction in training time while maintaining accuracy.} Additionally, we provide a framework for identifying invariant infinitesimal generators for arbitrary Cauchy problems. This unsupervised algorithm does not require prior numerical solutions, making it practical and efficient for various applications.}




\maketitle

\section{Introduction}\label{sec1}
% background of deep learning
{Partial differential equations (PDEs) govern numerous physical, biological, and engineering phenomena, including fluid dynamics, heat transfer, electromagnetism, and structural mechanics. Since most real-world PDEs lack closed-form solutions, numerical methods are essential for approximating their solutions. In recent years, the increasing interest in data-driven models has led to the emergence of a new field known as \textbf{scientific machine learning} (SciML), the combination of scientific computing and machine learning. One typical example of SciML is the solution of the differential equations using neural networks, dating back to 1990s~\cite{lagaris1998artificial}.}

% Popular methods to solve PDE by DL
{Several scientific machine learning methods have been developed, successfully tackling complex problems. For solving PDEs, Raissi et al. introduced \textbf{Physics-Informed Neural Networks (PINNs)} \cite{raissi2019physics}, which enforce physical laws as optimization constraints, enabling unsupervised training. As a mesh-free approach, PINNs flexibly handle complex boundaries and solve high-dimensional problems at arbitrary points. Various enhanced versions have been proposed to improve accuracy and efficiency.  Wang et al. \cite{wang2022and} introduced an adaptive weighting strategy to balance loss components based on the neural tangent kernel’s convergence rate, incorporating a causal parameter for sequential learning \cite{wang2024respecting}. McClenny et al. \cite{mcclenny2020self} developed \textbf{self-adaptive PINNs} using an adversarial mechanism to update loss weights. More recently, Anagnostopoulos et al. \cite{anagnostopoulos2024residual} proposed a \textbf{residual-based attention mechanism} with a decay parameter to regulate weight updates. For uncertainty quantification, \textbf{Bayesian-PINN} \cite{yang2021b} and \textbf{ensemble PINN} \cite{zou2024neuraluq} assess uncertainty from noisy or incomplete data. Zou et al. \cite{zou2024correcting} introduced a general approach to correct misspecified physical models in PINNs for discovering governing equations from sparse, noisy data.  Beyond PINNs, the \textbf{deep Ritz method} \cite{yu2018deep} minimizes the energy functional based on the solution's variational representation. Inspired by the Galerkin method, the \textbf{Deep Galerkin Method (DGM)} \cite{sirignano2018dgm} replaces basis functions with a neural network and employs a probability-based sampling strategy.  

}

{
% Open problem of fundamental solution
Despite the effectiveness of the aforementioned algorithms, they are limited to solving specific initial or boundary value problems. To obtain real-time responses for varying sources or boundary values, several kernel regression methods have recently been proposed, such as \textbf{graph kernel networks} \cite{anandkumar2020neural} and \textbf{Fourier neural networks} \cite{li2020fourier}.  For Cauchy initial value problems, the common mathematical foundation involves boundary integral representations using a \textbf{fundamental solution (FS)}. The solution to the PDE can be obtained by convolving the initial condition with the fundamental solution. However, solving the FS presents several challenges: (1) No explicit formula exists for the majority of PDEs. (2) Singularity points appear around the initial distribution.  
(3) The computational dimension of FS is $2n+1$, where $n$ is the state dimension. These issues lead to computational difficulties, limiting both efficiency and accuracy—especially in vanilla PINNs. Therefore, developing a fast algorithm that maintains high precision is essential.

% Lie symmetry group
To this end, we will leverage the powerful tool developed by Sophus Lie, \textbf{Lie group theory}, which can be employed to study the symmetries inherent in PDEs. By identifying these symmetries, complex PDEs can be simplified and solved systematically. Applications of Lie symmetry groups in PDEs include simplifying equations with additional first-order PDEs, finding invariant solutions, classification, structural analysis, and constructing new solutions. Recently, several promising approaches have emerged that combine symmetric properties with PINNs, particularly in the search for invariant solutions \cite{akhound2024lie, zhang2023symmetry, zhang2023enforcing}. For fundamental concepts and a technical review of their application to PDEs, we refer readers to \cite{olver1993applications}.

% Our work
In this work, we focus on solving the fundamental equation of the Cauchy problem using deep neural networks and Lie group theory. The initial $\delta$ distribution satisfies the invariance principle in a certain subspace of the symmetry groups of the system. Therefore, this symmetric property imposes an additional constraint on the first-order PDE for the fundamental solution, known as the \textbf{linearized symmetric condition (LSC)} \cite{nail2009practical, olver1993applications}. By incorporating the LSC as a new residual loss into the vanilla PINN, we shift large computational costs from the original physics residual to the first-order residual. This significantly speeds up the training process, as demonstrated by our numerical simulations. Our method not only retains the mesh-free advantage of the vanilla PINN but also offers the following contributions:

% the contribution of this paper
% \begin{enumerate} 
% \item A new symmetric residual loss is introduced to PINNs based on Lie symmetry group theory. 
% \item Numerically, our algorithm is simple and efficient and can significantly speed up the training process while enhancing the precision of the solution. Computational load from higher-order derivatives in residual loss can be partly transferred to first-order derivatives in symmetric loss.
% \item Our algorithm is unsupervised, meaning no prior training data are required. Furthermore, our method can be applied to high-dimensional PDEs, where computational acceleration is particularly significant. 
% \item A theoretical analysis is conducted to identify the symmetric subgroup and the linearized symmetric condition for arbitrary Cauchy problems. 
% \end{enumerate}
{
\begin{itemize}
    \item \textbf{Derivative-Order Reduction:} While adaptive PINNs \cite{wang2022and} optimize training through loss reweighting, we uniquely reduce the \textit{intrinsic computational complexity} of the PDE itself. By incorporating the Linearized Symmetric Condition (LSC) as a new residual term, we systematically replace high-order derivatives in the physics loss with first-order symmetry constraints (Theorem 1). This theoretically grounded approach could be further combined with adaptive weighting strategies in future work.

    \item \textbf{Deterministic Efficiency:} Unlike Bayesian-PINNs \cite{yang2021b} that focus on uncertainty quantification, our symmetry exploitation provides guaranteed computational gains. The LSC framework yields significant training acceleration while preserving accuracy (Section 4), as verified across multiple benchmark problems.

    \item \textbf{Computational Symmetry Exploitation:} Prior symmetry-based works \cite{akhound2024lie} primarily seek invariant solutions. Our key innovation lies in leveraging symmetries specifically for \textit{computational acceleration} - the LSC transfers optimization burden to first-order constraints without requiring known invariant solutions. This unsupervised approach is particularly valuable for high-dimensional problems where traditional PINNs struggle with derivative computations.
\end{itemize}

}

% organization of the paper
The remainder of the paper is organized as follows. Section 2 presents basic concepts on fundamental solutions and PINNs. A deep fundamental solution solver based on the invariance principle and PINNs will be derived in section 3. Section 4 provides several numerical examples to demonstrate the effectiveness and accuracy of the proposed algorithm.
}

\section{Basic concepts and Preliminaries}
At the beginning, some notations used in the paper will be presented for the convenience of the readers.

{\textbf{Notations:}
The set of real numbers is denoted by \(\mathbb{R}\). \(\mathbb{R}^k\) refers to the \(k\)-dimensional Euclidean space. Let \(C^{\infty}(U)\) be the set of differentiable functions defined on \(U\). \(\nabla \phi := \left(\frac{\partial \phi}{\partial x_1}, \cdots, \frac{\partial \phi}{\partial x_n}\right)^T\) denotes the gradient operator. \(\Delta \phi := \sum_{i=1}^{n} \frac{\partial^2 \phi}{\partial x_i^2}\) denotes the Laplacian operator. $\partial \Omega$ represents boundary of region $\Omega$. $\text{supp}(f)$ denotes support set of a function $f$. The inner product between two functions \(f\) and \(g\) is represented by \(\langle f, g \rangle := \int_U f g \, dx\).}

\subsection{Physics-Informed neural networks}
In the following, a general time-dependent PDE can be described as follows:
\begin{equation}\label{PDE_model}
\begin{cases}
\mathcal{L}_{t,x}[u(t,x)] = f(t, x),\quad (t,x)\in[0, T]\times\Omega\\
\mathcal{B}_{t,x}[u(t,x)] = g(t,x), \quad x\in\partial\Omega
\end{cases}
\end{equation}
{where \(\mathcal{L}_{t,x}\) and \(\mathcal{B}_{t,x}\) are two operators defining equation and boundary condition. \(u(t,x)\) is the desired exact solution. \(\Omega\) denotes the computational domain with its corresponding boundary $\partial\Omega$. $f$ and $g$ are defined as forcing source and boundary conditions. For the initial boundary value problem (IBVP), the initial condition can also be regarded as a special boundary condition.

In the context of PINNs, a DNN $u_{\theta}(t,x)$ is constructed to approximate the exact solution $u(t,x)$ where $\theta$ represents the parameters of the DNN. Compared to traditional supervised learning manner, PINNs are designed to incorporate the physical constraints in loss function with the help of automatic differentiation. More precisely, PINNs are optimized by encompassing data and physical loss.
\begin{equation}
\begin{aligned}
\mathcal{L}(\theta) =& \mathcal{L}_{data}(\theta) + \mathcal{L}_{physics}(\theta)
\end{aligned}
\end{equation}
where
\begin{equation}
\mathcal{L}_{data}(\theta) = \frac{w_d}{N_d}\sum_{i=1}^{N_d}|u_\theta(t_d^i,x_d^i) - u_d^i|^2
\end{equation}
and physical loss consists of boundary and residual loss,
\begin{equation}
\begin{aligned}
\mathcal{L}_{physics}(\theta) =& \mathcal{L}_b(\theta) + \mathcal{L}_r(\theta)\\
=& \frac{w_b}{N_b}\sum_{i=1}^{N_b}|\mathcal{B}_{t,x}[u_\theta(t_b^i,x_b^i)] - g(t_b^i,x_b^i)|^2 + \frac{w_r}{N_r}\sum_{i=1}^{N_r}|\mathcal{L}_{t,x}[u_\theta(t_r^i, x_r^i)] - f(t_r^i, x_r^i)|^2
\end{aligned}
\end{equation}
where $\{w_d, w_b, w_r\}$ are corresponding weights that can be adjusted by balancing different parts of loss values. $(t_s^i,x_s^i, u_d^i)$ are available dataset of solution with number $N_d$. $(t_b^i, x_b^i), (t_r^i, x_r^i)$ are sampled collocation points from the boundary and inner computational domain with number $N_b, N_r$. The DNN $u_\theta(t,x)$ parameterized by $\theta$ can be trained by Adam or L-BFGS optimizer.
}

\subsection{Cauchy problem}
We shall consider the Cauchy problem of the parabolic PDE taking the following form:
\begin{equation}\label{model}
\left\{
\begin{aligned}
&\pd{u}{t}(t,x) = \mathcal{N}_x u(t,x),\\
&u(0,x) = u_0(x),
\end{aligned}
\right.
\end{equation}
{where $(t,x) \in [0, T]\times \R^n$ and $\mathcal{N}_x$ denotes a linear differential operator defining the governing equation. $u_0(x)$ represents the initial condition.

In a Cauchy initial value problem (IVP) for a partial differential equation (PDE), the solution is determined by specifying initial conditions on an entire non-characteristic hypersurface (such as $t=0$ for time-dependent problems). This is fundamentally different from boundary value problems, where conditions must be imposed on the spatial boundaries. Corresponding to general form of PDE in Eq. (\ref{PDE_model}), two differential operators can be specified as
\begin{equation}
\begin{aligned}
\mathcal{L}_{t,x}[u] \overset{\Delta}{=}& \frac{\partial}{\partial t} - \mathcal{N}_x(u)\\
\mathcal{B}_{t,x}[u] =& u
\end{aligned}
\end{equation}
and $f(t,x) = 0, g(t,x) = u_0(x)$.

In order to solve the initial value problem (IVP) of Eq. (\ref{model}), in the following we shall introduce a method of fundamental solution which is a kernel method.}{
\begin{definition}[Fundamental solution]
For Cauchy problem Eq. (\ref{model}), the fundamental solution is denoted by  $K(t,x,y)$ satisfying
\begin{equation}\label{IVP_FS}
\begin{aligned}
&\pd{K}{t}(t,x,y) = \mathcal{N}_x K(t,x,y),\\
&\lim_{t\to 0}\int K(t,x,y)u_0(y)dy =  u_0(x).
\end{aligned}
\end{equation}
where $(t, x, y)\in [0,T]\times \R^n\times \R^n$. 
\end{definition}

Notice that the second equation is frequently written as $\lim_{t\to 0}K(t,x,y) = \delta(x-y)$ where $\delta$ is the usual delta function.} Fundamental solutions play a crucial role in determining the unique solution to the Cauchy problem. Once a fundamental solution is obtained, the Cauchy problem with different initial conditions can be directly solved using convolution.
\begin{equation}
u(t,x) = \int_{}K(t,x,y)u_0(y)dy
\end{equation}

% It is especially notable that if the explicit solution of \( K(t,x,y) \) at a neighborhood of \( t \in (0,\varepsilon] \) is obtained, where \( \varepsilon \) is a fixed small number, then the explicit solution for large time \( T \) can be expressed as
% {
% \begin{equation}
% \begin{aligned}
% K(T,x,y) =& \int_{x_1}\int_{x_2}\cdots\int_{x_M}K\left(\frac{T}{M},x,x_1\right)K\left(\frac{T}{M},x_1,x_2\right)\cdots\\
% & K\left(\frac{T}{M},x_M,y\right)dx_1dx_2\cdots dx_M
% \end{aligned}
% \end{equation}}
% where $M$ is the smallest integer greater than $\frac{T}{\varepsilon}$. The M-fold integrals are calculated in $x_1,x_2,\cdots,x_M$.

%###############
\section{Method}

\subsection{Invariance principle}
%In the previous section, the general theory of Lie symmetry group is introduced on the differential equation. In practice, additional initial and boundary information can be incorporated into the above general framework which leads to a critical tool as the invariance principle.

% brief introduction of Invariance principle
{In order to find inherent symmetry for the boundary value problem, boundary information should be taken into account in the Lie symmetry. The following statement plays an important role in determining the invariant condition of the boundary value problem (BVP) which is called the ``Invariance principle". }

\begin{definition}\label{IP}
Consider PDE $\mathcal{L}(x, u) = f(x)$ admitting a symmetry group $G$. The Dirichlet BVP
\begin{equation}\label{BVP}
\mathcal{L}(x, u) = f(x),\quad u|_S = h(x)
\end{equation} is said to be invariant under subgroup $H\subset G$ if 
\begin{enumerate}[(1)]
\item Boundary manifold $S$ is invariant under $H$.
\item Boundary condition $h(x)$ is invariant under $\tilde{H} := H|_S$.
\end{enumerate}
\end{definition}

{For the parabolic type Cauchy problem:
\begin{equation}\label{F-Sol}
\begin{aligned}
&K_t + \mathcal{N}_x[K](t,x,y) = 0
\end{aligned}
\end{equation}
where temporal spatial variable $(t,x)\in[0,T]\times\R^n$. The symmetry group is generated by a set of linearly independent infinitesimal generators 
\begin{equation}
G=\left<X_1, X_2, \cdots, X_m\right>
\end{equation}
{where $X_i$ represents a tangent vector field taking the following form:}
\begin{equation}
X = \tau(t,x,K)\p_t + \xi(t,x,K)\p_x + \eta(t,x,K)\p_K
\end{equation}
The symmetry group can be derived by solving the following equation
\begin{equation}\label{symmetry_equation}
pr^{(d)}(X)(\pd{K}{t} - \mathcal{N}_x[K](t,x,y)) = 0,\quad \text{when}\quad \pd{K}{t} - \mathcal{N}_x[K](t,x,y)=0
\end{equation}
where 
\begin{equation}
pr^{(d)}(X) = X + \sum_{|\alpha|\le d}\phi^{\alpha}(x, K^{(\alpha)})\partial_{K_\alpha}
\end{equation}
denotes the prolongation operator of $X$ which contains additional components of partial derivative such as $\partial_{K_t}, \partial_{K_x}, \partial_{K_{t,x}}, \cdots$. $\alpha = (\alpha_1, \alpha_2, \cdots, \alpha_n)$ denotes the multi-index. $\phi^\alpha(x, K^{\alpha})$ can be calculated by general prolongation formula.
It can be explained intuitively that the symmetry group keeps the equation invariant in its solution region. After solving equation (\ref{symmetry_equation}), we shall obtain all symmetry groups $G$ associated with Eq. (\ref{F-Sol}).

According to the invariance principle, our goal is to find a subgroup $H\subset G$ keeping the initial value problem (IVP) Eq. (\ref{IVP_FS}) to be invariant. In Eq. (\ref{IVP_FS}), boundary manifold is $S := \{t=0\}\times\R^n$. Subgroup restricted to boundary manifold is denoted as $\tilde{H} = H|_{S} = H_{\{t=0\}\times\R^n}$. Followed by the invariance principle, the subgroup we want to find satisfies two specific conditions: 
\begin{enumerate}
\item Boundary manifold $S := \{t=0\}\times\R^n$ is invariant under $H|_{S}$. Geometrically, that means $H|_{S}(S)\subset S$.
\item Boundary condition $K(0,x,y) = \delta(x - y)$ is under $H_S$. That means 
\begin{equation}
H|_S(K|_{S} - \delta(x - y))|_{K|_{S} = \delta(x - y)} = 0
\end{equation}
\end{enumerate}
}
In order to deal with the Dirac function, the usual mathematical formulation regards it as a distribution defined on differentiable function space.
\begin{definition}[Distribution]
A distribution is a linear bounded functional $f$ defined on $C_0^\infty(\R^n)$ satisfying
\begin{equation}
\langle f,a\varphi_1 + b\varphi_2\rangle = a\langle f,\varphi_1\rangle + b\langle f,\varphi_2\rangle
\end{equation}
\end{definition}

The transformation of a distribution can be defined as follows.
\begin{definition}[Transformation of Distributions]
Consider point transformation $\bar{x} = Tx$ for $x,\bar{x}\in\R^n$. Given two distributions, transformation $f\to\bar{f}$ is defined by the following equation 
\begin{equation}
\langle f(Tx), \varphi(x)\rangle = \langle\bar{f}(\bar{x}), \varphi(T^{-1}\bar{x})\rangle
\end{equation}
\end{definition}

Let $f$ be a distribution function, $T_a$ be a one-parameter transformation on $\R^n$, and $a$ is a group parameter taking small value $|a|\ll 1$. Denote $\bar{x} = T_a(x)$ with component $\bar{x}^i = x^i + a\xi^i(x) + \mathcal{O}(a^2)$. Then the infinitesimal transformation of $f$ will be (\cite{nail2009practical})
\begin{equation}
\bar{f} = f - a(\nabla\cdot\xi)f + \mathcal{O}(a^2)
\end{equation}

% For the sake of simplicity, the theoretical results we shall present will be restricted to one-dimensional spatial space, i.e., $x\in\R$. However, all results will be consistent with arbitrary spatial dimensions and extension is straightforward to implement.

{Starting from two conditions derived from the Definition \ref{IP}, we shall determine the condition generators in the subgroup should satisfy. In the following first theorem, we assume that the infinitesimal generator has the most general form. That means coefficients in front of each component depend on all variables $t,x,u$. This theorem will give us conditions that coefficients should satisfy in the most general form. The detailed proof can be found in the Appendix \ref{apd_proof}.}
\begin{theorem}\label{Main_thm}
Assume the first-order infinitesimal generator has the form
\begin{equation}
\mathbf{v} = \tau(t,x,u)\p_t + \xi(t,x,u)\p_x + \eta(t,x,u)\p_u
\end{equation}
The generator keeps the following Cauchy problem invariant
\begin{equation}
\begin{cases}
u_t + \mathcal{N}_x[u(t,x)] = 0, (t,x)\in[0,T]\times\R\\
u(0,x) = f(x), x\in\R
\end{cases}
\end{equation}
if and only if 
\begin{equation}\label{sym_coeff_cond}
\begin{cases}
\tau(0,x,u) = 0\\
\eta(0,x,f(x)) + \xi_x(0,x,f(x))\cdot f(x) = 0\\
x + a\xi(0,x,u)\in \text{supp}(f),\quad \text{if}\quad x\in \text{supp}(f)
\end{cases}
\end{equation}
where $\text{supp}(f)$ denotes support set of function $f$.
\end{theorem}

{In the following, we shall make further assumptions on the form of coefficients of infinitesimal generators. The detailed proof can be found in the Appendix \ref{apd_proof}.}
\begin{theorem}\label{special_generator_Cauchy}
Assume the first-order infinitesimal generator has the form
\begin{equation}
\mathbf{v} = \tau(t,x)\p_t + \xi(t,x)\p_x + \phi(t,x)u\p_u
\end{equation}
The generator keeps the following Cauchy problem invariant
\begin{equation}
\begin{cases}
u_t + \mathcal{N}_x[u(t,x)] = 0, (t,x)\in[0,T]\times\R\\
u(0,x) = f(x), x\in\R
\end{cases}
\end{equation}
if and only if
\begin{equation}
\begin{cases}
\tau(0,x) = 0\\
[\phi(0,x) + \xi_x(0,x)]\Big|_{\T{supp}(f)} = 0\\
x + a\xi(0,x,u)\in supp(f),\ \text{if}\ x\in supp(f)
\end{cases}
\end{equation}
where $\text{supp}(f)$ denotes support set of function $f$.
\end{theorem}

{Finally, the fundamental equation will be considered equipped with the Dirac function as the initial condition. Followed by the previous two theorems, the condition satisfied by the coefficient in the generator is derived to keep the Cauchy problem itself invariant. The detailed proof can be found in the Appendix \ref{apd_proof}.}
\begin{theorem}\label{special_generator_green}
Consider Cauchy problem of the fundamental solution
\begin{equation}\label{}
\begin{aligned}
&K_t + \mathcal{N}_x[K(t,x,y)] = 0, (t,x,y)\in[0,T]\times\R^2\\
&K(0,x,y) = \delta(x-y), (x,y)\in\R^2
\end{aligned}
\end{equation}
and generator
\begin{equation}
\mathbf{v} = \tau(t,x;y)\p_t + \xi(t,x;y)\p_x + \phi(t,x;y)u\p_u
\end{equation}
keep the above Cauchy problem invariant if and only if 
\begin{equation}
\begin{cases}
\tau(0,x;y) = 0\\
[\phi(0,x;y) + \xi_x(0,x;y)]\Big|_{x = y} = 0\\
\xi(0,y;y) = 0
\end{cases}
\end{equation}
\end{theorem}
{
For high-dimensional problems, the corresponding result can be directly extended to
\begin{theorem}\label{special_generator_green}
Consider Cauchy problem of the fundamental solution
\begin{equation}\label{}
\begin{aligned}
&K_t + \mathcal{N}_x[K(t,x,y)] = 0, (t,x,y)\in[0,T]\times\R^n\times\R^n\\
&K(0,x,y) = \delta(x-y), (x,y)\in\R^n\times\R^n
\end{aligned}
\end{equation}
and generator
\begin{equation}
\mathbf{v} = \tau(t,x;y)\p_t + \sum_{i=1}^n\xi_i(t,x;y)\p_{x_i} + \phi(t,x;y)u\p_u
\end{equation}
keep the above Cauchy problem invariant if and only if 
\begin{equation}
\begin{cases}
\tau(0,x;y) = 0\\
[\phi(0,x;y) + \nabla\cdot\xi(0,x;y)]\Big|_{x = y} = 0\\
\xi_i(0,y;y) = 0,\quad i=1,2,\cdots,n
\end{cases}
\end{equation}
where $\xi(t,x;y):=[\xi_1(t,x;y), \xi_2(t,x;y), \cdots, \xi_n(t,x;y)]^\top$.
\end{theorem}
}
\subsection{Deep symmetric fundamental solver}
Once the invariant subgroup related to the fundamental solution is found, the following LSC condition will be satisfied by the fundamental solution $K(t,x,y)$. \cite{olver1993applications}, i.e.,
\begin{equation}
Q(t,x,y,K) \overset{\Delta}{=} \eta(t,x,y,K) - \xi(t,x,y,K)K_x - \tau(t,x,y,K)K_t = 0
\end{equation}
It can be observed that this equation only involves the first-order derivative. This inherent symmetric property satisfied by fundamental solution inspires us to take it into account in vanilla PINN architecture. The fundamental solution parametrized by a neural network is denoted as $K_\theta(t,x,y)$. In order to deal with the initial singularity numerically, a simple and common technique is to approximate it by a Gaussian distribution with small variance \cite{cortez2001method, teng2022learning}. Empirical loss function involved symmetric property is defined as
\begin{equation}
\mathcal{L}(\theta) =  
\mathcal{L}_0(\theta) + \mathcal{L}_r(\theta) + 
\mathcal{L}_s(\theta)
\end{equation}
where
\begin{equation}
\begin{aligned}
&\mathcal{L}_0(\theta) = \frac{1}{\added{N_0}}\sum_{i=1}^{N_0}|K_\theta(0,x_0^i, y_0^i) - {\rho}(x_0^i, y_0^i)|^2\\
&\mathcal{L}_r(\theta) = \frac{1}{\added{N_r}}\sum_{i=1}^{N_r}|K_{\theta,t}(t_r^i,x_r^i,y_r^i) + \mathcal{N}_x[K_\theta(t_r^i,x_r^i,y_r^i)]|^2\\
& \mathcal{L}_s(\theta) = \frac{1}{\added{N_s}}\sum_{i=1}^{N_s}|Q(t_s^i,x_s^i,y_s^i,K_\theta(t_s^i,x_s^i,y_s^i))|^2
\end{aligned}
\end{equation}
{Similar to the setup described in the introduction for the vanilla PINN, $\{x_0^i, y_0^i\}$ are the initial value points sampled from the initial condition $\{t=0\}\times\Omega$. The set $\{t_r^i, x_r^i, y_r^i\}$ denotes the residual points sampled from the interior of the region $(0, T)\times\Omega$. The set $\{t_s^i, x_s^i, y_s^i\}$ represents the symmetric collocation points sampled from the same feasible region as the residual points. $N_0$, $N_r$, and $N_s$ represent the number of collocation points for the initial condition, PDE residual, and symmetric residual, respectively. The function $\rho(x,y) := \mathcal{N}(0, \sigma I)$ defines a multi-dimensional Gaussian density with variance parameter $|\sigma|\ll 1$. }

{To ensure stable training and improve accuracy, we employ practical numerical strategies in this study. We use the Adam optimizer and L-BFGS to train the neural network, dynamically adjusting the learning rate for faster convergence. The learning rate $\eta_k$ at each iteration $k$ is determined by the formula $\eta_k = \eta_0 \times \gamma^k$, where $\eta_0$ is the initial learning rate, $\gamma$ is a regularization factor, and $\eta_k$ is the learning rate at iteration $k$. Typically, each iteration consists of 50 gradient descent steps, although this number can be adjusted as needed. Additionally, we set a predefined stopping criterion $e_{stop}$ for the training loss; once the loss reaches this threshold, training halts to save time.}

{
\begin{remark}
Our method fundamentally differs from existing PINN enhancements in both objective and mechanism: 
(i) Unlike \textit{adaptive PINNs} \cite{wang2022and} that optimize training dynamics through loss reweighting, we extended the framework of PINN by adding the a new residual term in the loss function based on symmetry, and this reduce the \textit{intrinsic computational complexity} of the PDE problem itself via derivative-order reduction. Similarly, for new residual terms, an adaptive weighting strategy can also be considered which is beyond the scope of this study;
(ii) In contrast to \textit{Bayesian-PINNs} \cite{yang2021b} that quantify uncertainty, we provide deterministic efficiency gains through symmetry exploitation; 
(iii) While prior \textit{symmetry-based approaches} \cite{akhound2024lie} focus on invariant solutions, our LSC framework uniquely targets \textit{computational acceleration} by transferring the optimization burden to first-order constraints. 
This theoretical distinction translates to consistent 20-30\% efficiency gains while maintaining accuracy, as demonstrated in Section 4.
\end{remark}}

% section 4
\section{Numerical results}
{In this section, we test the numerical simulation based on the proposed symmetric PINN of the Green kernel (GSPINN) on three typical parabolic PDEs. The first equation we considered here is the (1+1)-dimensional heat equation. The second example is the (1+1)-dimensional spatially varying diffusion equation where the coefficient of the second-order derivative depends on the spatial variable. Finally, we tested the performance of GsPINN in dealing with high-dimensional equations by using a (1+2)-dimensional heat equation. Correspondingly, the computational dimensions of the fundamental equation are 3, 3, and 5 in three examples.

As we will demonstrate, GsPINN can significantly speed up the optimization process of the vanilla PINN while maintaining higher accuracy. All simulations are implemented on a laptop with a 12th Gen Intel(R) Core(TM) i9-12900H 2.50 GHz processor.}

\begin{example}[3D FS of heat equation]
{The heat equation describes the distribution of temperature (or more generally, the distribution of a diffusing quantity) over time within a given space. In its simplest form, the (1+1)-dimensional heat equation is formulated as below:}
\begin{equation}
\pd{u}{t}(t,x) = \frac{1}{2}\pdd{}{x}u(t,x),\quad (t,x)\in[0,T]\times\R^n
\end{equation}
where $u: [0,T]\times\R\to\R$ is desired solution to be solved. Spatial and temporal dimensions are both 1.
\end{example}

By definition, its corresponding fundamental solution satisfies the equation in the following:
\begin{equation}
\begin{aligned}
&\pd{K}{t}(t,x,y) = \frac{1}{2}\pdd{K}{x}(t,x,y)\\
&K(0,x,y) = \delta(x-y), \quad(x,y)\in\R^2
\end{aligned}
\end{equation}
where $\delta(\cdot)$ denotes a multi-dimensional Dirac function. This Cauchy problem will have the exact solution which is also called the heat kernel.
\begin{equation}
\begin{aligned}
K(t,x;y)
=& (2\pi t)^{-\frac{1}{2}}\exp(-\frac{(x-y)^2}{2t})
\end{aligned}
\end{equation}{
By applying the Lie symmetry method, the following infinitesimal generator can be verified to preserve both the initial manifold and the delta function invariance.
\begin{equation}
v = 2t\frac{\p}{\p t} + x\frac{\p}{\p x} - nu\frac{\p}{\p u}
\end{equation}
Its corresponding characteristic can be obtained as
\begin{equation}
Q = -u - (x-y)\pd{u}{x} - 2t\pd{u}{t}
\end{equation}
Therefore, the invariant principle implies that the fundamental solution will have the symmetry invariant.
\begin{equation}
K(t,x,y) + (x-y)\pd{K}{x}(t,x,y) + 2t\pd{K}{t}(t,x,y) = 0
\end{equation}}
{In the following, detailed network parameters will be listed. For both PINN and GsPINN, a fully connected network is chosen with 3 inputs and 1 output and 10 hidden layers of width 20 in each layer. Initialization of network parameters will be implemented by Xavier initializer. Latin hypercubic sampling strategy will be used to produce different parts of collocation points. Different loss weights are set as $w_i=w_r = w_s = 1$. The starting learning rate is chosen as 1e-3 and will decay after each 1000 epoch. In terms of the optimizer, Adam will be first used to train followed by L-BFGS in which the maximum iteration step is set to 50000.}

\begin{figure}[!hbt]
\centering
\includegraphics[width=6cm,height=5cm]{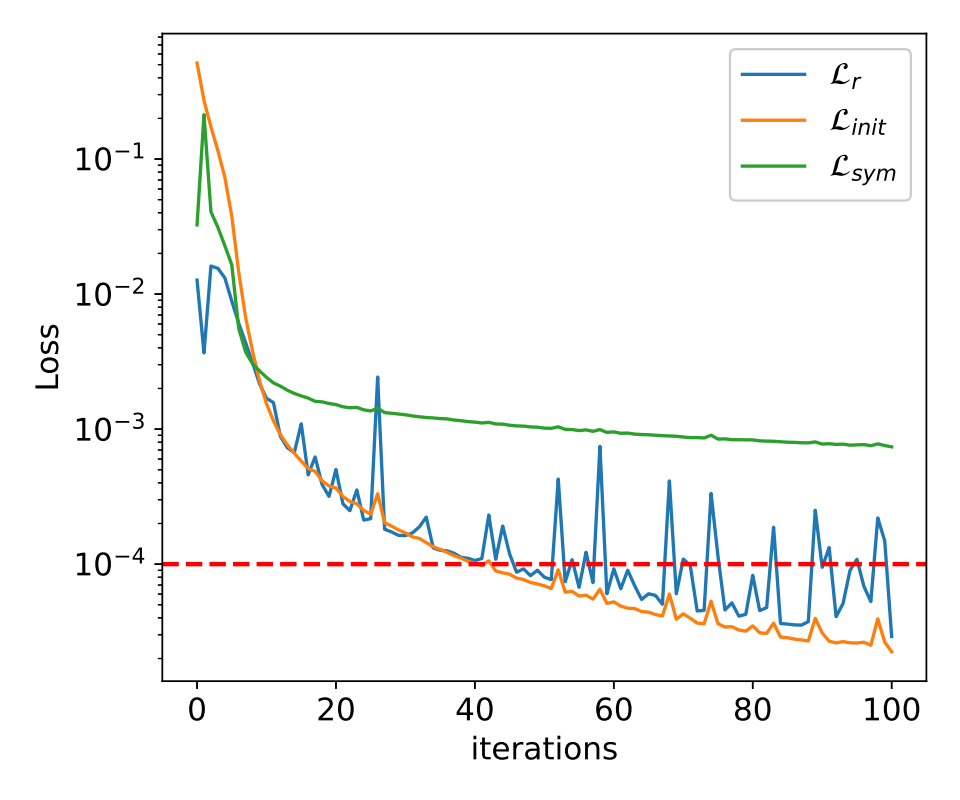}
\includegraphics[width=6cm,height=5cm]{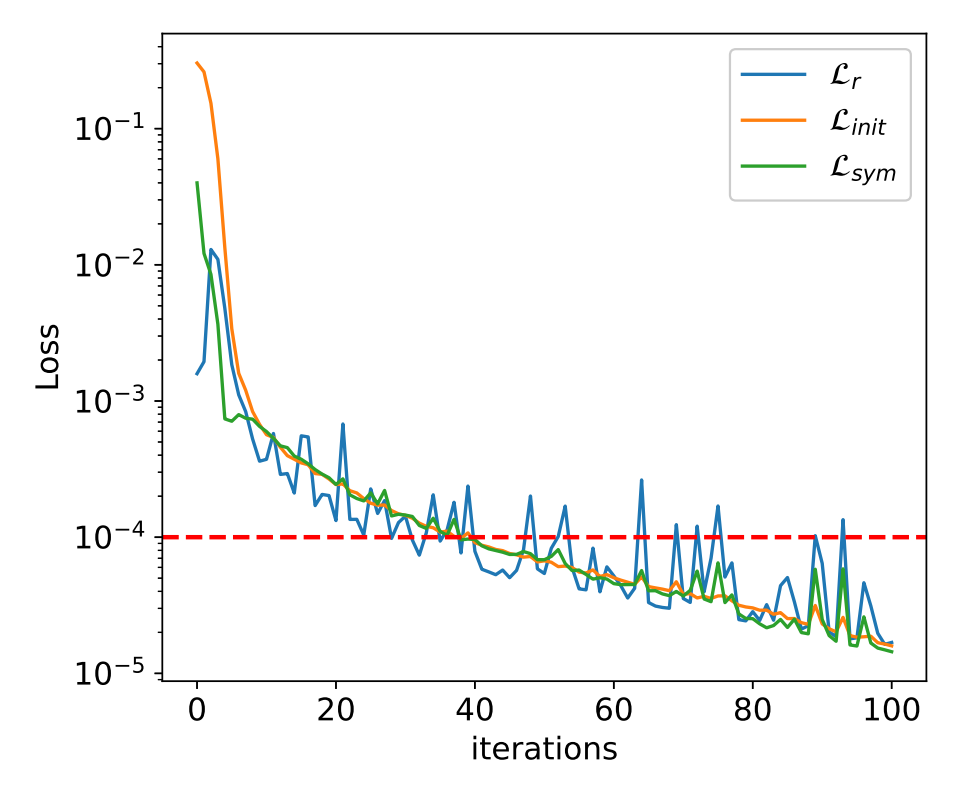}
\caption{Loss value of vanilla PINN (left) and GsPINN (right) in {Example 1}.}\label{ex1_loss}
\end{figure}

% Explain the performance of numerical simulations
{
Fig. \ref{ex1_loss} presents the evolution of three different components of the loss function. In this implementation, initial and residual collocation points are specified as 1000 and 5000 for both PINN and GsPINN. Symmetric points for GsPINN are set as 5000. It can be observed that, with the assistance of GsPINN, the newly introduced symmetry-based loss decreases rapidly and ultimately reaches a lower value. In Table \ref{ex1_table}, we conduct a systematic comparison between vanilla PINN and GsPINN, evaluating the accuracy of different loss components under various sets of collocation points. The results highlight the crucial role of GsPINN, which significantly reduces computational time while simultaneously improving accuracy. For the combination $[N_0, N_r, N_s] = [1000,5000,5000]$, computational time has been saved up to 24\% and accuracy still remains the same with PINN of the order of magnitude 1e-3.

Fig. \ref{ex1_PINN} illustrates the numerical solution at a specific time step and its comparison with the standard solution, clearly demonstrating the improved solution precision.}

\begin{table}[!hbt]
\scriptsize
\caption{Influence of number in collocation points in {Example 1}. Adam(3000) + L-BFGS, Loss weights $\lambda_i = \lambda_r = \lambda_s = 1.0$}
\label{ex1_table}
\centering
\begin{tabular}{c|c|c|c|c|c|c}
\hline
Algorithm & Loss(MSE) & Loss(init) & Loss(res) & Loss(sym) & time & points($[N_i,N_r,N_s]$) \\
\hline
PINN & 2.657e-03 & 3.562e-05 & 6.230e-05 & / & 611.5450 & [1000, 5000, 0] \\
\hline
GsPINN & \textbf{2.604e-03} & 1.309e-04 & 2.391e-04 & 8.917e-04 & 459.8366 & [1000, 5000, 5000] \\
\hline
GsPINN & 5.981e-03 & 7.560e-05 & 3.576e-04 & 2.155e-04 & \textbf{271.8132} & [1000, 3000, 5000] \\
\hline
\hline
PINN & 3.328e-03 & 3.629e-05 & 9.299e-05 & / & 515.7466 & [1000, 10000, 0] \\
\hline
GsPINN & 2.719e-03 & 1.100e-05 & 9.412e-05 & 3.192e-04 & 469.6560 & [1000, 10000, 3000] \\
\hline
GsPINN & 6.043e-03 & 9.164e-05 & 4.827e-04 & 1.628e-04 & 381.6296 & [1000, 10000, 5000] \\
\hline
GsPINN & 3.318e-03 & 3.229e-05 & 7.835e-05 & 1.890e-04 & 551.7782 & [1000, 10000, 10000] \\
\hline
GsPINN & 4.539e-03 & 2.329e-04 & 1.510e-03 & 5.336e-04 & 444.1592 & [1000, 5000, 10000] \\
\hline
GsPINN & 5.068e-03 & 2.702e-05 & 1.307e-04 & 9.137e-04 & 553.8613 & [1000, 3000, 10000] \\
\hline
\end{tabular}
\end{table}

\begin{figure}
\centering
\subfigure[]{\includegraphics[width=6cm,height=4.5cm]{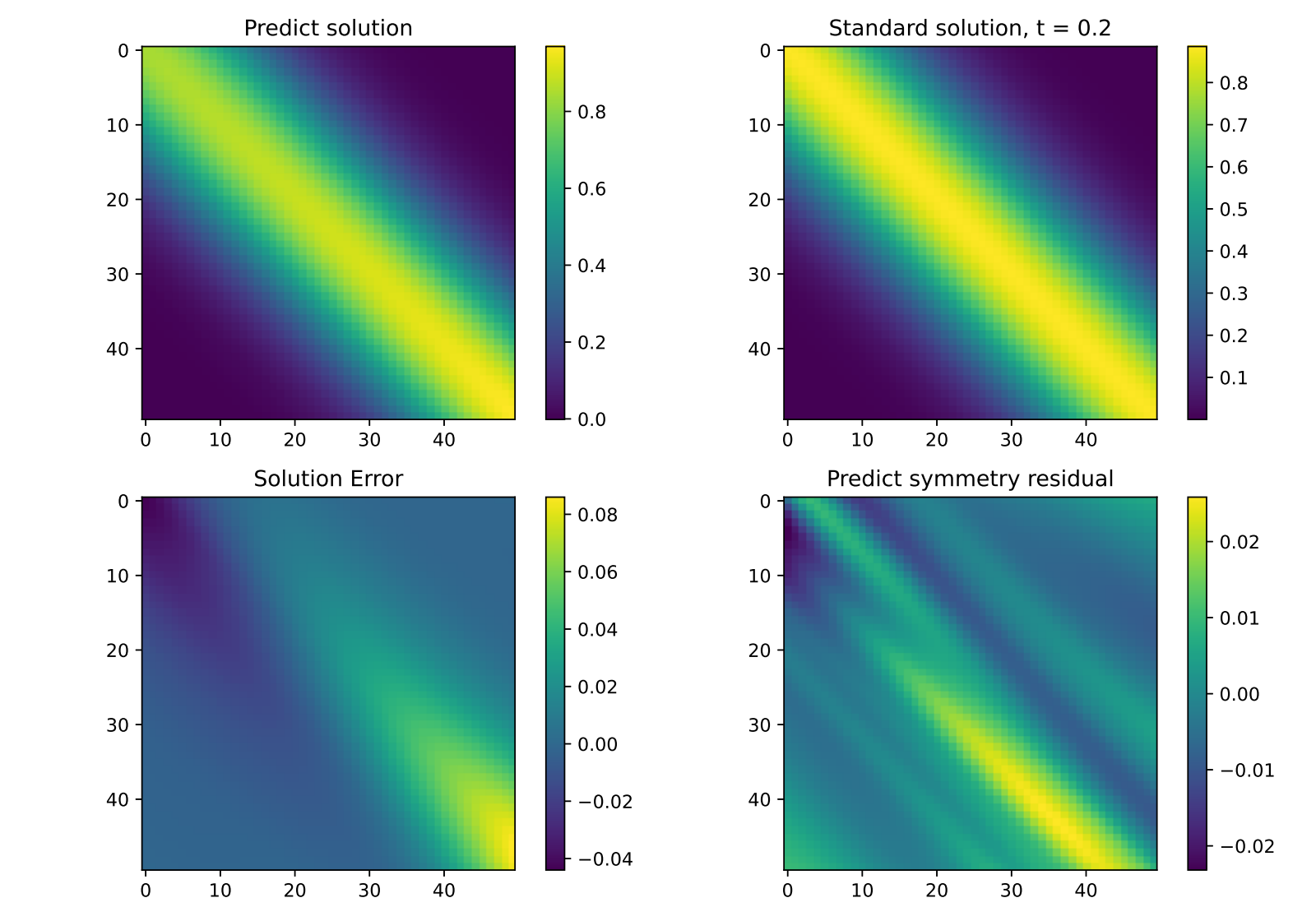}}
\subfigure[]{\includegraphics[width=6cm,height=4.5cm]{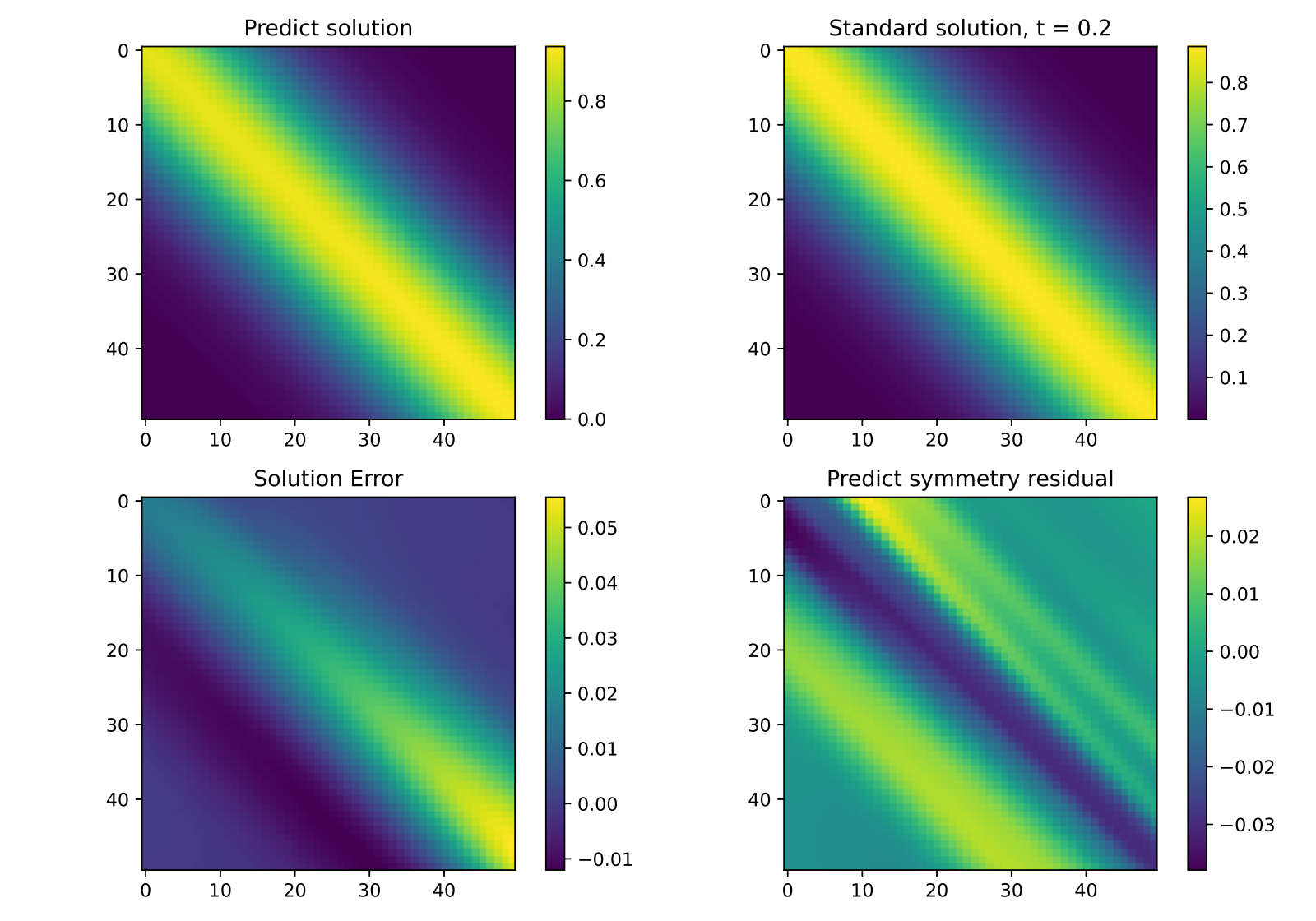}}
\caption{The comparison between PINN and GsPINN with the exact solution in {Example 1}. (a) We present the predicted solution by PINN, the exact solution, the \(L_2\) error, and the predicted symmetry residual, respectively. (b) correspond to GsPINN, with the same meaning.}\label{ex1_PINN}
\end{figure}

\begin{example}[3D FS of spatially varying diffusion equation]
{The spatially varying diffusion equation describes the process of diffusion in a medium where the diffusivity varies across space. It is a generalization of the classical heat equation and takes the form:}
\begin{equation}
\pd{u}{t}(t,x) = x\pdd{u}{x}(t,x) + b\pd{u}{x}(t,x),\quad (t,x)\in[0,T]\times\R_+
\end{equation}
where diffusion coefficient $b$ is set 0.5. Domain of spatial-temporal variable $(t,x)$ is $[0,T]\times\R_+$. $u: [0,T]\times\R_+\to\R$ is the desired solution.
\end{example}
It is easy to find diffusion equation above is linear so it is associated with a fundamental solution.
\begin{equation}
\begin{aligned}
&\pd{K}{t}(t,x;y) = x\pdd{K}{x}(t,x;y) + \frac{1}{2}\pd{K}{x}(t,x;y)\\
&K(0,x;y) = \delta(x-y)
\end{aligned}
\end{equation}
where $K: [0,T]\times\R_+\times\R_+\to\R$ is fundamental solution. The initial condition is set as a delta function which can be regarded as a point source.

This Cauchy problem admits the explicit solution:
\begin{equation}
K(t,x;y) = \frac{1}{\sqrt{\pi ty}}\exp(-\frac{x+y}{t})\cosh(\frac{2\sqrt{xy}}{t})
\end{equation}{
By applying the Lie symmetry method, the following infinitesimal generator can be verified to preserve the invariance of both the initial manifold and the delta function.}
\begin{equation}
v = t^2\frac{\p}{\p t} + 2tx\frac{\p}{\p x} - (x-y+bt)u\frac{\p}{\p u}
\end{equation}
Therefore, the invariant principle implies that the fundamental solution will have the symmetry invariant.
\begin{equation}
(x-y+\frac{1}{2}t)K(t,x,y) + t^2\pd{K}{t}(t,x,y) + 2tx\pd{K}{x}(t,x,y) = 0
\end{equation}

{In this example, the majority of detailed network parameters are the same as the first examples. The only difference is that loss weights are adjusted by using $w_i = 10, w_r = w_s = 1.0$.}

\begin{table}[!hbt]
\scriptsize
\caption{Influence of number in collocation points in Example 2. Adam(3000) + L-BFGS, Loss weights $\lambda_{i}=\lambda_{r}=\lambda_{s}=1.0$}
\label{ex2_table}
\centering
\begin{tabular}{c|c|c|c|c|c|c}
\hline
Algorithm & Loss(MSE) & Loss(init) & Loss(res) & Loss(sym) & time & points([$N_{i},N_{r},N_{s}$]) \\ 
\hline
PINN & 3.612e-03 & 1.204e-05 & 3.290e-05 & / & 581.2354 & [500, 10000, 0] \\ 
GsPINN & \textbf{3.489e-03} & 1.020e-04 & 3.672e-04 & 1.459e-04 & 386.7697 & [500, 10000, 3000] \\ 
GsPINN & 3.532e-03 & 1.000e-04 & 3.530e-04 & 1.568e-04 & 413.6504 & [500, 10000, 5000] \\ 
GsPINN & 3.668e-03 & 6.768e-05 & 3.625e-04 & 1.409e-04 & 454.1149 & [500, 10000, 10000] \\ 
GsPINN & 4.008e-03 & 5.213e-05 & 2.528e-04 & 1.531e-04 & \textbf{372.9199} & [500, 5000, 10000] \\  
\hline
PINN & 4.327e-03 & 1.038e-05 & 3.372e-05 & / & 625.3559 & [1000, 10000, 0] \\ 
GsPINN & 3.610e-03 & 1.300e-04 & 4.149e-04 & 1.626e-04 & \textbf{350.9773} & [1000, 10000, 3000] \\ 
GsPINN & \textbf{3.387e-03} & 6.989e-05 & 3.610e-04 & 1.302e-04 & 398.2439 & [1000, 10000, 5000] \\ 
GsPINN & 3.550e-03 & 4.710e-05 & 2.621e-04 & 1.275e-04 & 601.7399 & [1000, 10000, 10000] \\ 
GsPINN & 3.993e-03 & 3.854e-05 & 1.442e-04 & 1.166e-04 & 407.6808 & [1000, 5000, 10000] \\ 
GsPINN & 3.928e-03 & 5.342e-05 & 2.155e-04 & 1.038e-04 & 388.8910 & [1000, 3000, 10000] \\ 
\hline
\end{tabular}
\end{table}

\begin{figure}[!hbt]
\centering
\includegraphics[width=10cm,height=7cm]{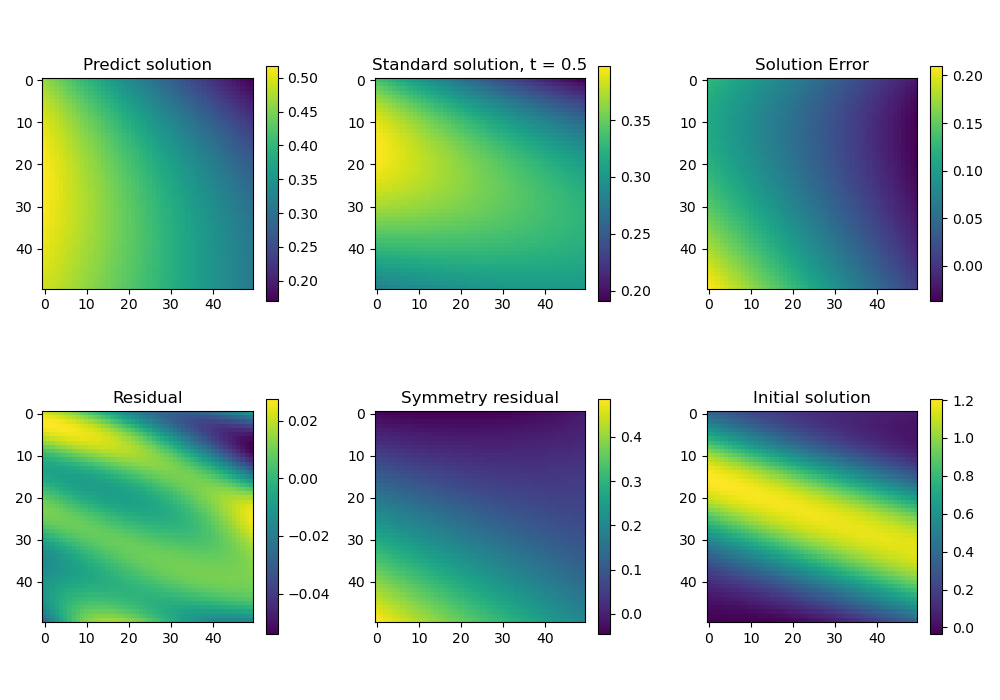}
\caption{Comparison between exact solution and PINN solution in {Example 2}.}\label{ex2_PINN}
\end{figure}

\begin{figure}[!hbt]
\centering
\includegraphics[width=10cm,height=7cm]{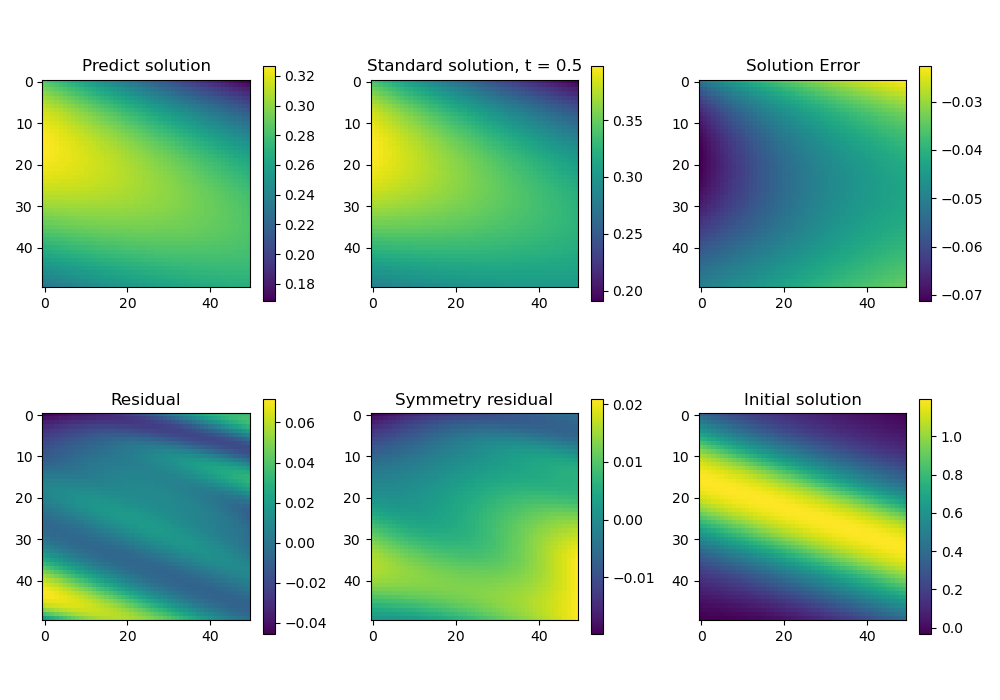}
\caption{Comparison between the exact solution and GSPINN solution in {Example 2}.}\label{ex2_GsPINN}
\end{figure}

\begin{figure}[!hbt]
\centering
\includegraphics[width=6cm,height=5cm]{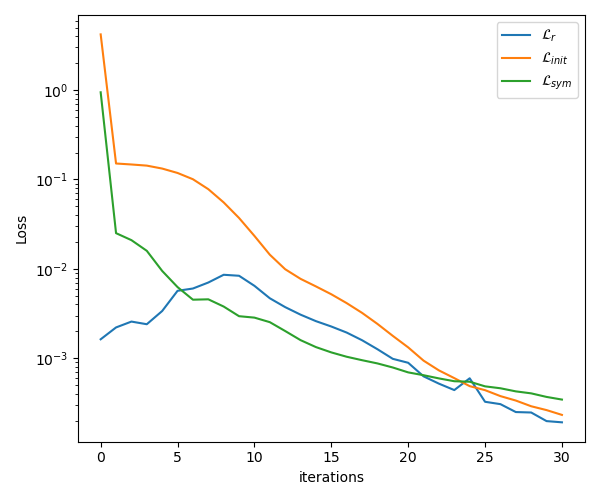}
\includegraphics[width=6cm,height=5cm]{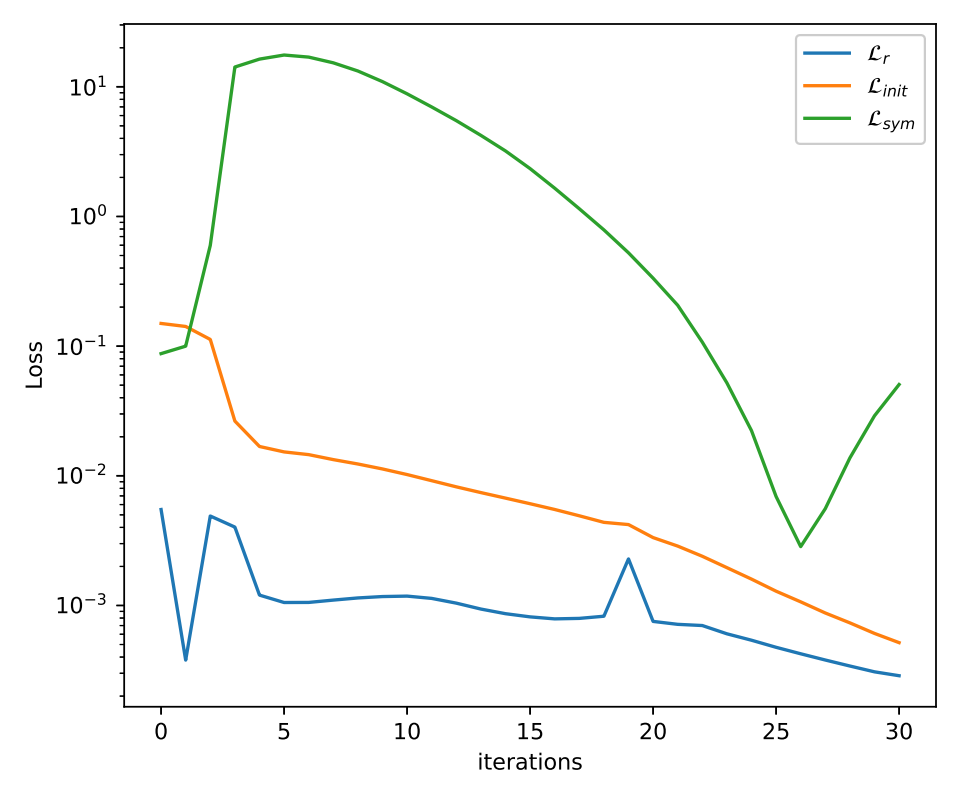} 
\caption{Loss value of GsPINN (left) and vanilla PINN (right) in {Example 2}.}\label{ex2_loss}
\end{figure}

Fig. \ref{ex2_loss} presents the loss components across three distinct categories. {In this implementation, initial and residual collocation points are specified as 500 and 10000 for both PINN and GsPINN. Symmetric points for GsPINN are set as 3000.} The results show that with the incorporation of GsPINN, the symmetry-related loss decreases rapidly and reaches a lower value compared to the others. Table \ref{ex2_table} systematically compares the accuracy of different loss components between vanilla PINN and GsPINN under varying collocation points, highlighting GsPINN’s ability to significantly reduce computational time while simultaneously improving accuracy. {For the combination $[N_0, N_r, N_s] = [500,10000,3000]$, computational time has been saved up to 33.5\% and accuracy still remains the same with PINN of the order of magnitude 1e-3.} Fig. \ref{ex2_PINN} and Fig. \ref{ex2_GsPINN} illustrate the numerical solutions at a specific time step alongside the standard solution, clearly demonstrating the improved precision achieved by our method.

{
\begin{example}[5D FS of heat equation]
The (1+2)-dimensional heat equation considered here is formulated as below:
\begin{equation}
\pd{u}{t}(t,x) = \frac{1}{2}(\pdd{}{x_1}u(t,x) + \pdd{}{x_2}u(t,x)),\quad (t,x)\in[0,T]\times\R^n
\end{equation}
where $u: [0,T]\times\R^2\to\R$ is desired solution to be solved. Spatial and temporal dimensions are $2$ and 1 respectively.
\end{example}

By definition, its corresponding fundamental solution satisfies the equation in the following:
\begin{equation}
\begin{aligned}
&\pd{K}{t}(t,x,y) = \frac{1}{2}(\pdd{K(t,x,y)}{x_1} + \pdd{K(t,x,y)}{x_2})\\
&K(0,x,y) = \delta(x-y), \quad(x,y)\in\R^2\times\R^2
\end{aligned}
\end{equation}

where $\delta(\cdot)$ denotes a multi-dimensional Dirac function. This Cauchy problem will have the exact solution which is also called the heat kernel.
\begin{equation}
\begin{aligned}
K(t,x,y)
=& (2\pi t)^{-1}\exp\left(-\frac{(x_1-y_1)^2 + (x_2-y_2)^2}{2t}\right)
\end{aligned}
\end{equation}
By applying the Lie symmetry method, the following infinitesimal generator can be verified to keep the initial manifold and delta function invariant.
\begin{equation}
v = 2t\p_t + (x_1\p_{x_1} + x_2\p_{x_2}) - nu\p_u
\end{equation}
Its corresponding characteristic can be obtained as
\begin{equation}
Q = -2u - [(x_1-y_1)\p_{x_1}+(x_2-y_2)\p_{x_2}]u - 2t\p_t u
\end{equation}
Therefore invariant principle implies that the fundamental solution will have the symmetry invariant.
\begin{equation}
2K(t,x,y) + [(x_1-y_1)\pd{}{x_1} + (x_2-y_2)\pd{}{x_2}]K(t,x,y) + 2t\pd{}{t} K(t,x,y) = 0
\end{equation}
}

{In this example, the architecture of a neural network is designed to have 5 input neurons 1 output neuron, and 10 hidden layers with 20 neurons in each layer. Besides, loss weights in this example are set as $w_i = w_s = 1.0$ and $w_r = 5.0$. The other network parameters related to training are the same as the first two examples}

In Table 3, we conduct a more systematic simulation to compare the vanilla PINN and GsPINN in terms of loss precision across different sets of collocation points. The table highlights the significant advantage of GsPINN which can take into account the symmetric information. Computational time has been saved up to 21\% and accuracy still remains the same with PINN of the order of magnitude 1e-4. Symmetry condition inspired loss component plays an remarkable role in dealing with fundamental solution compared to vanilla PINNs. 

\begin{table}[h]\scriptsize
\centering
\caption{Comparison of Different Algorithms in Example 3 (Selected Results)}\label{tab:comparison3}
\begin{tabular}{c|c|c|c|c|c|c}
\hline
Algorithm & L2 error & Loss(init) & Loss(res) & Loss(sym) & Time (s) & [$N_0, N_r, N_s$] \\
\hline
PINN & 9.824e-04 & 3.035e-04 & 1.248e-04 & / & 384.1212 & [2000, 10000, 0] \\
GspINN & \textbf{6.493e-04} & \textbf{2.170e-04} & 7.026e-05 & \textbf{9.006e-04} & \textbf{304.2987} & [3000, 10000, 10000] \\
\hline
\end{tabular}
\end{table}

\section{Conclusion}
In this paper, we aim to propose a novel and fast algorithm to approximate the fundamental
solution of linear differential equation. The motivation of this paper is to utilize the Lie group theory combined with a Physics-Informed neural network to achieve the speedup of computation. First of all, we shall compute the symmetry group admitted by the fundamental equation without any initial condition. After obtaining a symmetry group, based on the invariance principle, a subgroup will be determined to keep the initial manifold and distribution unchanged. Corresponding conditions that should be satisfied have been present in the form of Theorems in section 3. Based on the subgroup, a linearized symmetric condition (LSC) can be derived. Then LSC is applied in the architecture of PINN to derive our novel algorithm termed GsPINN. Detailed numerical experiments are present including a (1+2)-dimensional spatially-varying diffusion equation and (1+4)-dimensional heat equation. Our novel algorithm has several advantages compared to vanilla PINN. Numerically, GsPINN can significantly speed up the network training procedure while keeping even enhancing the precision of the solution. GsPINN exhibits the capability to effectively handle high-dimensional fundamental equation. Furthermore, our novel algorithm has strong theoretical support to identify invariant subgroups.

%\section{Data Availability Statement}  
%The demo code supporting this study will be made available if requested.

\section{Acknowledgement}
This work was supported by the National Natural Science Foundation of China (No.12501613, No. 42404128)

\begin{appendices}

\section{Additional background of group theory}\label{secA1}
Next, we shall introduce some elementary knowledge about group theory and Lie algebra theory. In order to be suitable for more general readers, concrete examples will be illustrated under each abstract definition. 

As a beginning, an abstract formulation of the group will be given as a most fundamental concept.

\begin{definition}[Group]
A group $G$ is a set equipped with an binary operation $\circ: G\times G\to G$ satisfying following property:
\begin{enumerate}
\item (identity) There exists identity element $e$ such that $e\circ g = g\circ e = g$ holds for any $g\in G$.
\item (associativity) For any elements $g, h, w\in G$, $(g\circ h)\circ w = g\circ (h\circ w)$ holds.
\item (inverse) For any element $g\in G$, there exists element $g^{-1}\in G$ satisfying $g\circ g^{-1} = g^{-1}\circ g = e$.
\end{enumerate}
\end{definition}
\begin{example}
The most common cases of group include $(\mathbb{Z}, +), (\mathbb{Q}\backslash\{0\}, *)$.
\end{example}

\begin{definition}[Lie group]
A Lie group is a smooth manifold associated with a group structure such that group operations are smooth. More precisely, group multiplication $\circ: G\times G\to G$ and inversion mapping are needed to be differentiable.
\end{definition}
\begin{example}
Consider $GL(n,\R)$ is the set of all invertible matrices of size $n\times n$ equipped with usual matrix multiplication. Notice that each component of multiplication and inversion are differentiable so that $(GL(n,\R), \circ)$ is a Lie group.
\end{example}

\begin{definition}[Group action]
Let $G$ be a Lie group. A group action is defined as a mapping $\cdot: G\times M\to M$ on manifold $M$ which satisfies
\begin{enumerate}
\item (Identity) $e\cdot a = a$ holds for any $a\in M$.
\item (Compatibility) $g\cdot(h\cdot a) = (g\circ h)\cdot a$ holds for any $g,h\in G$ and $a\in M$.
\end{enumerate}
\end{definition}

Next, we give some basic concepts related to Lie algebra.
\begin{definition}
If $X$ and $Y$ are differential operators, the Lie bracket of $X$ and $Y$, $[X, Y]$, is defined by $[X, Y]\phi = X(Y\phi) - Y(X\phi)$ for any $C^\infty$ function $\phi$.
\end{definition}

\begin{definition}
A vector space $\mathcal{F}$ with the Lie bracket operation $\mathcal{F}\times \mathcal{F}\rightarrow \mathcal{F}$ denoted by $(x, y)\longmapsto [x, y]$ is called a Lie algebra if the following axioms are satisfied

{\rm (1)} The Lie bracket operation is bilinear;

{\rm (2)} $[x, x]=0$ for all $x\in\mathcal{F}$;

{\rm (3)} $[x, [y, z]] + [y, [z, x]] + [z, [x, y]] = 0,\quad x, y, z\in\mathcal{F}$.
\end{definition}

\subsection{Lie symmetry method}

Lie group, as a powerful tool in mathematics, originated from the study of the symmetric property of solution set, especially for the solution of PDE. We shall consider a general form PDE system $\mathcal{L}(x, u) = 0$ with its solution space denoted as $\T{Sol}(\mathcal{L})$. Here $x\in\R^p$ is independent coordinates, $u = u(x)$ is the unique solution and $\mathcal{L}$ represents a differential operator.

\textbf{Reformulation of PDE system} In the following, another viewpoint of PDE will be introduced with the assumption of order $n$. First in order to describe all derivative variables of the solution, jet space $J^{(n)}$ is introduced as $J^{(n)}:= X\times U\times U_1\times \cdots\times U_n$. Here $X = \R^p$ is domain of independent variable $x$, and $$U_i := \left\{\frac{\p^i u}{\p x_{j_1}\p x_{j_2}\cdots \p x_{j_p}}: j_1+j_2+\cdots j_p = i\right\}$$ represent set of all $i$-order partial derivative. 

Therefore, solution space $\T{Sol}(\mathcal{L})$ can be understood as a subvariety $S_\Delta := \{(x, u^{(n)})\in J^{(n)}: \mathcal{A}(x,u^{(n)}) = 0\}$, where $\mathcal{A}: J^{(n)}\to\R$ is a smooth function mapping. A simple example can be shown if considering the Laplace equation $u_{x_1x_1} + u_{x_2x_2} = 0$ with independent variable space $\R^2$. Correspondingly independent variable $x = (x_1, x_2)$, and jet expansion of dependent variable $u^{(2)} = (u, u_{x_1}, u_{x_2}, u_{x_1x_1}, u_{x_1x_2}, u_{x_2x_2})\in\R^6$. Smooth mapping $\mathcal{A}: (x, u^{(2)})\mapsto u_{x_1x_1} + u_{x_2x_2}$ is an affine function. Next in order to describe symmetric transformation, a symmetry group is introduced.
\begin{definition}[Symmetry group]
A symmetry group related to the PDE system refers to a local group action $G\curvearrowright X\times U$ which satisfies for any $g\in G$, $u = g\cdot f\in\T{Sol}(\mathcal{A})$ if $u = f(x)\in\T{Sol}(\mathcal{A})$. Elements in symmetry group are usually called transformations.
\end{definition}
It is noted that a transformation has nothing to do with derivative of a solution $f$. In order to extend the action to the jet space, we shall introduce a concept of prolongation. First of all, the following simple example is shown to illustrate the transformation on a concrete function.
\begin{example}
Consider Lie group $SO(2)$ describing the set of rotation matrice of size $2\times 2$. The function mentioned here is an affine polynomial $u = ax + b$. Next transformed function with a rotation $\theta$ can be written down
\begin{equation}
\tilde{u} = \frac{\sin\theta + a\cos\theta}{\cos\theta - a\sin\theta}\tilde{x} + \frac{b}{\cos\theta - a\sin\theta}
\end{equation}
\end{example}

The notion of prolongation is motivated by extending the domain of a group action such that we can simultaneously make transformation of tangent space.

\begin{definition}[Prolongation of Group action]\label{prolong_GA}
let $M$ be an open set in $X\times U$ and $M^{(n)}:=M\times U_1\times \cdots\times U_n$. Given a function $u = u(x)$ with corresponding graph defined in $M$, the prolongation of group transformation $g$ is defined as
\begin{equation}
\T{pr}^{(n)}g\cdot (x, u^{(n)}) := (\tilde{x}, \tilde{u}^{(n)})
\end{equation}
where
\begin{equation}
\tilde{u}^{(n)} := \T{pr}^{(n)}(g\cdot u)(\tilde{x})
\end{equation} 
\end{definition}
From the above definition, prolongation of group action is based on the prolongation of a single function. The main consideration lies in first applying $g$ to transform the function $u$ then obtaining derivative from new transformed function at the point $\tilde{x}$.

After finishing the prolongation of group action, corresponding infinitesimal generator can be also prolonged to high order jet space. Let $v$ be a infinitesimal generator defined on $M\subset X\times U$ with one-parameter flow $\exp(\varepsilon v)$. Infinitesimal generators we considered here are all differentiable. This leads to an equivalent relation between infinitesimal generator and one-parameter flow. The motivation here is any smooth infinitesimal generator $v$ will induce a differentiable flow $\psi(\epsilon, x) = \exp(\epsilon v)x$ in which group action $\exp(\varepsilon v)$ can be detached simultaneously. One-parameter group action can be prolonged by applying method defined in Definition \ref{prolong_GA}. This new prolonged action will correspond to a new flow followed by a new prolonged infinitesimal generator $\T{pr}^{(n)v}$, i.e.,
\begin{equation}
pr^{(n)}v = \frac{d}{d\varepsilon}\Big|_{\varepsilon=0}pr^{(n)}\exp(\varepsilon v)(x, u^{(n)})
\end{equation}

In practice, a general prolongation formula can be applied to extend infinitesimal generator to arbitrary order. For the sake of simplicity, we shall only consider the case of $q = 1$. 
\begin{proposition}[General prolongation formula]
Let 
\begin{equation}
v = \sum_{i=1}^{p}\xi^i(x,u)\p_{x_i} + \phi(x,u)\p_{u}
\end{equation}
be a zero-order infinitesimal generator. Its corresponding prolongation of order $n$ can be written as 
\begin{equation}
pr^{(n)}v = v + \sum_{|J|\le n}\phi^{J}(x, u^{(n)})\p_{u_J}
\end{equation}
where component $\phi^{J}(x, u^{(n)}) = D_J(\phi - \sum_{i=1}^{p}\xi^i\p_{x_i}(u)) + \sum_{i=1}^{p}\xi^i\p_{x_i}(u_J)$, $D_J$ denotes total derivative.
\end{proposition}
\textbf{Difference between partial and total derivative.}

\begin{example}
Let infinitesimal generator defined as $v = -u\p_x + x\p_u$ with corresponding $\xi = -u$ and $\phi = x$. By general prolongation formula, $pr^{(2)}v = v + \phi^x\p_{u_x} + \phi^{xx}\p_{u_{xx}}$ where components $\phi^x, \phi^{xx}$ are calculated as below,
\begin{equation}
\begin{aligned}
\phi^x =& D_x(x + uu_x) - uu_{xx}\\
=& 1 + u_x^2 + uu_{xx} - uu_{xx}\\
=& 1 + u_x^2
\end{aligned}
\end{equation}
and
\begin{equation}
\begin{aligned}
\phi^{xx} =& D_{xx}(x + uu_x) + (-u)u_{xxx}\\
=& D_x(1 + u_x^2 + uu_{xx}) - uu_{xxx}\\
=& 3u_xu_{xx}
\end{aligned}
\end{equation}
\end{example}
Up to now, we can calculate prolonged infinitesimal generators of arbitrary order which can allow us to obtain the following linearized symmetry condition (LSC).
\begin{definition}[LSC]
For a prolonged infinitesimal generator,
\begin{equation}
X = \xi\p_x + \tau\p_t + \eta\p_u + \eta^x\p_{u_x} + \eta^t\p_{u_t} + \cdots
\end{equation}
Linearized symmetry condition states $X(\mathcal{A}(x, u^{(n)}))$ will vanish at the solution set $\mathcal{A}(x, u^{(n)}) = 0$, i.e.,
\begin{equation}
X(\mathcal{A}(x, u^{(n)})) = 0,\quad\T{when } \mathcal{A}(x, u^{(n)}) = 0
\end{equation}
\end{definition}
Solving LSC will yield a set of linearly independent infinitesimal generators $\{v_i\},i\ge 1$ which span a Lie algebra named symmetry algebra. There is a one-to-one correspondence between symmetry algebra and symmetry group.

\section{Detailed Proofs}\label{apd_proof}

Here, some detailed proof will be provided.

\textbf{Proof of Theorem \ref{Main_thm}.}

Notice that the initial manifold here is $S = \{t=0\}\times\R$. First requirement of invariance principle is that $v(t)\Big|_S = 0$, where $(\cdot)$ means a functional action over function $t$. It implies component should be satisfied by $\tau(0,x,u) = 0$. Followed by this requirement, infinitesimal generator $\mathbf{v}$ restricted on initial manifold will become
\begin{equation}
\mathbf{v}\vl_S = \xi(0,x,u)\p_x + \eta(0,x,u)\p_u
\end{equation}
The second requirement of invariance principle is stated as $\mathbf{v}\vl_S(u - f)\vl_{u = f} = 0$. By definition of infinitesimal transformation, we shall derive $\bar{x} = x + a\xi(0,x,u) + \mathcal{O}(a^2)$ and $\bar{u} = u + a\eta(0,x,u) + \mathcal{O}(a^2)$. By utilizing Lemma \ref{lem_infi_trans}, initial distribution will be transformed as $\bar{f} = f - a\xi_x(0,x,u)\cdot f$. It will yield that 
\begin{equation}
\begin{aligned}
\mathbf{v}\vl_S(u - f)\vl_{u = f} &= \mathbf{v}\vl_S(u)\vl_{u = f} - \mathbf{v}\vl_S(f)\vl_{u = f}\\
&= [\bar{u} - \bar{f}]\vl_{u = f}\\
&= a[(\eta(0,x,u) + \xi_x(0,x,u)\cdot f)]\vl_{u = f} + \mathcal{O}(a^2)\\
&= a[\eta(0,x,f(x)) + \xi_x(0,x,f(x))\cdot f(x)] + \mathcal{O}(a^2)
\end{aligned}
\end{equation}
Finally, $\mathbf{v}\vl_S(u - f)$ vanishes at the $u=f$ is equivalent to $\eta(0,x,f(x)) + \xi_x(0,x,f(x))\cdot f(x) = 0$ for $x\in\R$.

Finally it should be noted that symmetry group should keep support of initial distribution unchanged to satisfy the invariance principle. Considering generator
\begin{equation}
v\vl_{t=0} = \xi(0,x,u)\partial_x + \eta(0,x,u)\partial_u
\end{equation}
is corresponding to transformation $g: \R\to\R$ by $x\mapsto\bar{x}$ under infinitesimal transformation. Invariance of support is equivalent to $g(supp(f))\subset supp(f)$, i.e.,
\begin{equation}
x + a\xi(0,x,u)\in supp(f),\ \text{if}\ x\in supp(f)
\end{equation}
\hfill$\square$

\textbf{Proof of Theorem \ref{special_generator_Cauchy}.}

By applying Theorem \ref{Main_thm}, we can easily obtain $\tau(0,x) = 0$ from the first equation (\ref{sym_coeff_cond}). Furthermore, starting from second identity of (\ref{sym_coeff_cond}), 
\begin{equation}
\begin{aligned}
\eta(0, x, f(x)) + \xi_x(0,x)\cdot f(x) &= \phi(0,x)f(x) + \xi_x(0,x)\cdot f(x)\\
&= [\phi(0,x) + \xi_x(0,x)]f(x)\\
&= 0
\end{aligned}
\end{equation}
the above identity implies the desired result.
\hfill$\square$

\textbf{Proof of Theorem \ref{special_generator_green}.}

The only thing we need to do is to replace $f(x)$ with $\delta(x-y)$ and notice $\T{supp}(\delta) = \{x = y\}$. Results of Theorem \ref{special_generator_Cauchy} correspond to the following
\begin{equation}
\begin{cases}
\tau(0,x;y) = 0\\
[\phi(0,x;y) + \xi_x(0,x;y)]\Big|_{\T{supp}(\delta)} = 0
\end{cases}
\end{equation}
where $y$ should be regarded as an external parameter.
\hfill$\square$

\textbf{Derivation of symmetric condition for heat equation}

Recall the heat equation,
\begin{equation}
u_t -  u_{xx} = 0
\end{equation}
The corresponding symmetry group can be obtained by routing procedure mentioned in previous section.
\begin{equation}
\begin{aligned}
&\mathbf{v}_0 = \p_t,\quad \mathbf{v}_1 = \p_x,\\
&\mathbf{v}_2 = 2t\p_t + x\p_{x},\quad \mathbf{v}_3 = u\p_u,\quad \mathbf{v}_4 = 2t\p_{x} - xu\p_u\\
&\mathbf{v}_5 = t^2\p_t + tx\p_{x} - \frac{1}{4}(2nt + x^2)u\p_u
\end{aligned}
\end{equation}
\begin{theorem}
The following operator $\mathbf{v}_2 - n\mathbf{v}_3 + y\mathbf{v}_1$ is a invariant transformation admitted by Cauchy problem
\begin{equation}
\begin{aligned}
&\pd{K}{t}(t,x;y) = \pdd{K}{x}(t,x;y)\\
&K(0,x;y) = \delta(x-y), \quad(x,y)\in\R\times\R
\end{aligned}
\end{equation}
\end{theorem}

\begin{proof}
The only thing we need to do is to verify the condition of Theorem \ref{special_generator_green}. First we notice that both operators $\mathbf{v}_2, \mathbf{v}_3$ do not satisfy the invariant condition mentioned in Theorem \ref{special_generator_green}. In the following, we shall consider the new combination:
\begin{equation}
\mathbf{v}_2 - n\mathbf{v}_3 + y\mathbf{v}_1 = 2t\p_t + (x + \tilde{\lambda})\p_{x}  + \lambda u\p_u
\end{equation}
In this example, there are the following correspondence
\begin{equation}
\begin{aligned}
\tau(t,x;y) =& 2t\\
\xi(t,x;y) =& x + \tilde{\lambda}\\
\phi(t,x;y) =& \lambda
\end{aligned}
\end{equation}
Direct computation implies that
\begin{equation}
\begin{aligned}
\tau(0,x;y) \equiv 0
\end{aligned}
\end{equation}
and
\begin{equation}
[\phi(0,x;y) + \xi_x(0,x;y)]\Big|_{x=y} = \lambda + \nabla_x(x) = \lambda + n
\end{equation}
and
\begin{equation}
\xi(0, y;y) = y + \tilde{\lambda}
\end{equation}
Theorem \ref{special_generator_green} will hold if and only if $\lambda = -n$ and $\tilde{\lambda}=-y$ that is the desired result.
\end{proof}

\textbf{Derivation of symmetric condition for diffusion equation}

Recall the diffusion equation
\begin{equation}
u_t = xu_{xx} + bu_x,\quad b>0
\end{equation}
The corresponding symmetry group can be obtained by the routing procedure mentioned in the previous section.
\begin{equation}
\begin{aligned}
&\mathbf{v}_0  = u\p_u,\quad \mathbf{v}_1 = \p_t, \quad \mathbf{v}_2 = t\p_t + x\p_x\\
&\mathbf{v}_3 = t^2\p_t + 2tx\p_x - (x + bt)u\p_u
\end{aligned}
\end{equation}

\begin{theorem}
The following operator $\mathbf{v}_3 + \lambda\mathbf{v}_0$ is a invariant transformation admitted by Cauchy problem
\begin{equation}
\begin{aligned}
&\pd{K}{t}(t,x;y) = x\pdd{K}{x}(t,x;y) + \frac{1}{2}\pd{K}{x}(t,x;y)\\
&K(0,x;y) = \delta(x-y)
\end{aligned}
\end{equation}
\end{theorem}
\begin{proof}
The only thing we need to do is to verify the condition of Theorem \ref{special_generator_green}. First we notice that both operators $\mathbf{v}_0$ and $\mathbf{v}_3$ do not satisfy the invariance condition mentioned in Theorem \ref{special_generator_green}. In the following, we shall consider the new combination:
\begin{equation}
\mathbf{v}_3 + \lambda\mathbf{v}_0 = t^2\p_t + 2tx\p_x - (x + bt - \lambda)u\p_u
\end{equation}
In this example, there are the following correspondence
\begin{equation}
\begin{aligned}
&\tau(t,x;y) = t^2\\
&\xi(t,x;y) = 2tx\\
&\phi(t,x;y) = -(x + bt - \lambda)
\end{aligned}
\end{equation}
Direct computation implies that
\begin{equation}
\begin{aligned}
\tau(0,x;y) \equiv 0
\end{aligned}
\end{equation}
and
\begin{equation}
[\phi(0,x;y) + \xi_x(0,x;y)]\Big|_{x=y} = [-(x-\lambda) ]\Big|_{x=y} = \lambda - y
\end{equation}
and
\begin{equation}
\xi(0,x;y) = 0
\end{equation}
Theorem \ref{special_generator_green} will hold if and only if $\lambda = y$ that is the desired result.
\end{proof}



\end{appendices}


\bibliography{myref}% common bib file

\end{document}